\documentclass[12pt]{extarticle}
\usepackage{amsmath, amsthm, amssymb, mathtools, hyperref, color}
\usepackage[shortlabels]{enumitem}
\usepackage{graphicx}
\usepackage{adjustbox}
\usepackage{xspace}
\usepackage{cleveref}
\usepackage{tikz}
\usepackage{multirow}
\usepackage{booktabs}
\usepackage{array}
\usepackage{tabularx}
\usetikzlibrary{arrows,calc,fit,matrix,positioning,cd}
\usepackage{float}
\usepackage{natbib}
\usepackage{placeins}

\definecolor{cb-yellow}{RGB}{221,170,51}
\definecolor{cb-red} {RGB}{187,85,102}
\definecolor{cb-green}{RGB}{17,119,51}

\hypersetup{
    pdfauthor={Janike Oldekop},
    pdftitle={Euler Stratifications},
    pdfsubject={},
    pdfkeywords={},
    colorlinks=true,
    linkcolor=cb-red,
    filecolor=cb-red,      
    urlcolor=cb-green,
    citecolor=cb-green,
}

\oddsidemargin \evensidemargin
\newtheorem{theorem}{Theorem}
\newtheorem*{theorem*}{Theorem}
\numberwithin{theorem}{section}
\newtheorem{proposition}[theorem]{Proposition}
\newtheorem{lemma}[theorem]{Lemma}
\newtheorem{corollary}[theorem]{Corollary}
\newtheorem{conjecture}[theorem]{Conjecture}

\newtheorem*{notation*}{Notation}
\newtheorem*{outline*}{Outline}
\newtheorem*{symexchange*}{Symmetric exchange axiom}
\newtheorem*{acknowledgements*}{Acknowledgements}

\theoremstyle{definition}
\newtheorem{definition}[theorem]{Definition}

\newtheorem{example}[theorem]{Example}

\newcommand{\RR}{\mathbb{R}}

\newcommand{\ZZ}{\mathbb{Z}}
\newcommand{\NN}{\mathbb{N}}
\newcommand{\PP}{\mathbb{P}}
\newcommand{\CC}{\mathbb{C}}

\newcommand{\cF}{\mathcal{F}}
\newcommand{\cH}{\mathcal{H}}

\newcommand{\cM}{\mathcal{M}}

\newcommand{\cV}{\mathcal{V}}

\newcommand{\cK}{\mathcal{K}}

\DeclareMathOperator{\mldeg}{MLdeg}

\DeclareMathOperator{\rank}{rank}

\newcolumntype{C}[1]{>{\centering\arraybackslash}p{#1}}

\date{}

\title{\textbf{Euler Stratifications of Second Hypersimplices via Delta-matroids}}
\author{Janike Oldekop}

\begin{document}
\maketitle

\begin{abstract}
    We study Euler characteristics of scaled toric varieties arising from second hypersimplices. In algebraic statistics, these are closely connected to maximum likelihood (ML) degrees of toric models. We establish a correspondence between delta-matroids and the non-vanishing factors of the principal $A$-determinant, providing an explicit connection between delta-matroid theory and algebraic statistics. Using this framework, we show that a conjectured minimum ML degree is realizable by a suitable embedding of the variety. Furthermore, for second hypersimplices up to order six, we prove that this value is minimal among all embeddings, as conjectured by Clarke~et~al.~(2024).
\end{abstract}

\section{Introduction}
\label{sec:ToricModels}

The notion of Euler stratifications arises in both the study of scattering amplitudes in theoretical physics and the investigation of toric models in algebraic statistics~\cite{EulerStratificationsOfHypersurfaceFamilies}. These stratifications are closely connected to \emph{Euler discriminants}, which were first introduced by Esterov in the study of discriminants of polynomial systems~\cite{TheDiscriminantOfASystemOfEquations}. 

Given a family of parameterized quasi-projective varieties, an \emph{Euler stratification} partitions the parameter space into strata on which the Euler characteristic of the corresponding fibers is constant \cite[Definition 2.1]{EulerStratificationsOfHypersurfaceFamilies}. In this article, we study families of scaled toric varieties associated with second hypersimplices as described in the following subsection. 

For $d\in\NN$ with $d\ge2$, the \emph{second hypersimplex of order $d$} is the lattice polytope $$\Delta_{d,2} \coloneqq \{ x\in[0,1]^d \mid x_1+\ldots+x_d=2 \} \subseteq\RR^d.$$ 

From a statistical perspective, these varieties give rise to an important class of discrete statistical models \cite[Chapter 6]{AlgebraicStatistics}. Such models, known as \emph{toric models}, play a central role in likelihood geometry \cite{LikelihoodGeometry}.

\subsection{Likelihood Geometry of Toric Models} \label{subsection:LikelihoodGeometryOfToricModels}

A well-known method in statistics is \emph{maximum likelihood (ML) estimation}. In algebraic statistics \cite{AlgebraicStatistics}, a discrete statistical model is viewed as a subset of the probability simplex $$\Delta_{n-1} \coloneqq \{ (p_1, \ldots, p_n) \in \RR^n \mid p_1 + \ldots + p_n = 1 \textup{ and } p_i \ge 0 \textup{ for all } i \in [n] \},$$ where $[n] \coloneqq \{ 1,\ldots, n \}$. A point in $\Delta_{n-1}$ represents the probability distribution of a discrete random variable with $n$ states. To estimate such a distribution from observed data, suppose we perform a random experiment with $n$ possible outcomes over $N$ independent trials. Accordingly, we consider $N$ independent and identically distributed random variables $X_k$ taking values in $[n]$, with probabilities $\PP(X_k=i) = p_i$. We record the observed data in a vector $u\in\NN^n$, where the $i^{\textup{th}}$ component $u_i$ denotes the number of observations of outcome~$i$. Then $N=u_1+\ldots + u_n$, and the \emph{likelihood function} is the rational function $$L_u(p) \coloneqq \frac{p_1^{u_1} \cdots p_n^{u_n}}{(p_1+\ldots+p_n)^{u_1+\ldots+u_n}}.$$ 

In ML estimation, one seeks to determine the distribution in a given model $\cM\subseteq \Delta_{n-1}$ that maximizes the likelihood function over $\cM$. Such a distribution is the \emph{ML estimate} of the data. It can be determined by computing the critical points of $L_u$ restricted to $\cM$.

A \emph{toric models} is determined by a configuration of $n$ lattice points $a_1, \ldots, a_n \in\ZZ^d$. Let $$A = \begin{pmatrix} a_1 & \ldots & a_n \end{pmatrix} \in\ZZ^{d\times n}.$$ We assume that the all-ones vector lies in the row space of $A$, and that $\ZZ A = \ZZ^d$. For $c\in(\CC^*)^n$, consider the parameterization map $$\psi_A^c : (\CC^*)^d \to \CC\PP^{n-1}, \quad  (\theta_1,\ldots,\theta_d) \mapsto [c_1 \theta^{a_1}:\dots : c_n\theta^{a_n}],$$ where $\theta^{a_j} \coloneqq \theta_1^{a_{1j}} \cdots \theta_d^{a_{dj}}$. The Zariski closure of the image of $\psi_A^c$ is the \emph{scaled toric variety} associated with $A$ and $c$, and is denoted by $V_A^c$. The vector $c$ is called a \emph{scaling}.

\begin{definition}
A \emph{toric model} $\cM_A^c$ is the intersection of a scaled toric variety $V_A^c$ with the open probability simplex. More precisely, $\cM_A^c = \varphi (V_A^c) \cap \Delta_{n-1}^\circ$, where $$\varphi:\CC\PP^{n-1} \to \CC^n, \quad [p_1 : \cdots : p_n] \mapsto \frac{1}{p_1 + \cdots + p_n} (p_1,\ldots,p_n).$$
\end{definition}

Each such model is associated with a polytope, namely the convex hull of the defining points $a_1,\ldots,a_n$. This polytope encodes important information about the model. For instance, its normalized volume equals the degree of the underlying toric variety, which in turn bounds the number of complex critical points of the likelihood function for generic data. A more precise complexity measure for algebraic ML estimation is the \emph{ML degree} \cite{TheMaximumLikelihoodDegree}.

\begin{definition}
The \emph{ML degree} of a scaled toric variety $V_A^c$ is the number of complex critical points of the likelihood function for generic $u\in \CC^n$, restricted to $V_A^c\setminus\cH$, where $$\cH \coloneqq \{ p \in \CC\PP^{n-1} \mid p_1 \ldots p_n (p_1 + \cdots + p_n) =0 \}.$$
\end{definition}

We now focus on toric models arising from the second hypersimplex $\Delta_{d,2}$. Assigning a scaling $c_{ij}$ to each of its lattice points, $e_i+e_j$ for $1\le i<j\le d$, where $e_i\in\RR^d$ denotes the~$i^{\textup{th}}$ standard basis vector, yields a scaled toric variety. 

\begin{notation*}
    For $c\in(\CC^*)^{d(d-1)/2}$, the scaled toric variety associated with $\Delta_{d,2}$ is denoted by~$V_d^c$. 
\end{notation*}

One goal proposed in \cite[Section~3]{MatroidStratificationOfMLDegreesOfIndependenceModels} is to understand how the ML degree of $V_d^c$ depends on the scaling. For any choice of scalings, the degree of $V_d^c$ is $2^{d-1}-d$ \cite[Proposition 3.2]{MatroidStratificationOfMLDegreesOfIndependenceModels}. Considering generic scalings, the ML degree equals the degree of the variety, while for non-generic scalings an ML degree drop occurs \cite[Corollary~8]{TheMaximumLikelihoodDegreeOfToricVarieties}. The locus of non-generic scalings is described by the \emph{principal $A$-determinant} \cite[Chapter 10]{DiscriminantsResultantsAndMultidimensionalDeterminants}.

Following \cite[Section 3]{MatroidStratificationOfMLDegreesOfIndependenceModels}, we associate to a given scaling $c \in (\CC^*)^{d(d-1)/2}$ a symmetric matrix $C\in \CC^{d\times d}$ whose entries are defined by $$C_{ij} = \begin{cases} c_{ij}, & i< j, \\ 0, & i=j, \end{cases} \quad \textup{and} \quad C_{ij} = C_{ji}.$$ We refer to this matrix as the \emph{scaling matrix} of $V_d^c$, and obtain the following description.

\begin{theorem}[{\cite[Theorem~3.6]{MatroidStratificationOfMLDegreesOfIndependenceModels}}] \label{thm:principal-A-determinant}
For $d\ge4$, the principal $A$-determinant of $\Delta_{d,2}$ is $$\prod_{E\subseteq[d], \, \vert E\vert\ge 4} \det C[E],$$ where $C[E]$ denotes the principal submatrix of $C$ indexed by $E$.
\end{theorem}

This leads naturally to the problem of determining the minimum ML degree attainable by a suitable choice of scalings. We study this problem for $\Delta_{d,2}$ using a result by June Huh that allows the ML degree of toric varieties to be studied from a topological perspective.

\begin{theorem}[{\cite[Theorem~1]{TheMaximumLikelihoodDegreeOfAVeryAffineVariety}}] \label{thm:ML-degree-topological-Euler-characteristic}
Let $V \setminus \cH$ be a $d$-dimensional smooth variety. The ML degree of $V$ equals the signed Euler characteristic $(-1)^d \chi (V \setminus \cH)$.
\end{theorem}

\subsection{Results} \label{subsection:Results}

In light of Theorem \ref{thm:principal-A-determinant}, we study the ML degree of $V_d^c$ in terms of the vanishing principal minors of the associated scaling matrix $C\in\CC^{d\times d}$. In particular, we focus on scalings for which all factors of the principal $A$-determinant vanish, proving an initial observation presented in~\cite[Table 5]{MatroidStratificationOfMLDegreesOfIndependenceModels}. Our first result can be summarized~as~follows.

\begin{theorem} \label{thm:main-rank-three}
If $C\in\CC^{d\times d}$ has rank three, then the ML degree of $V_d^c$ is $$\mldeg(V_d^c) = \binom{d-1}{2}.$$ In particular, there exists a scaling $c \in (\CC^*)^{d(d-1)/2}$, for which this ML degree can be realized.
\end{theorem}

In \cite[Conjecture 3.9]{MatroidStratificationOfMLDegreesOfIndependenceModels}, the authors conjecture that the ML degree given in the above theorem cannot be reduced by any choice of scalings. We prove this conjecture for small $d$ by deriving explicit formulas for the ML degree of $V_d^c$, depending on the rank of $C$.

\begin{theorem} \label{thm:main-minimum}
For $2\le d\le6$, the minimum ML degree of $V_d^c$ that can be achieved by choosing a suitably scaling $c\in(\CC^*)^{d(d-1)/2}$ is $\binom{d-1}{2}$.
\end{theorem}

To prove Theorem~\ref{thm:main-minimum}, we establish a new connection between the ML degrees of scaled toric varieties arising from second hypersimplices and \emph{delta-matroids}. Introduced by Andr\`e Bouchet \cite[Section 6]{GreedyAlgorithmAndSymmetricMatroids}, these combinatorial structures are characterized by the symmetric exchange axiom, generalizing the basis exchange property of matroids.
To every scaling matrix we associate a delta-matroid via the non-singular principal submatrices of $C$. In this way, the delta-matroid encodes the vanishing pattern of the principal $A$-determinant. 

\begin{theorem} \label{thm:main-delta-matroid-invariance}
Let $2 \le d\le 6$. If $C_1, C_2\in\CC^{d\times d}$ are scaling matrices whose principal minors have the same vanishing pattern, up to simultaneous permutation of rows and columns, then $$\mldeg(V_d^{c_1}) = \mldeg(V_d^{c_2}).$$
\end{theorem}

This corresponds to the situation in independence models, where the ML degree is a matroid invariant \cite[Theorem 1.3]{MatroidStratificationOfMLDegreesOfIndependenceModels}, and suggests that the ML degree of $V_d^c$ is a delta-matroid invariant.

\begin{conjecture} \label{conjecture:main-delta-matroid-invariance}
Let $C_1, C_2 \in \CC^{d\times d}$ be scaling matrices. If the delta-matroids associated to $C_1$ and $C_2$ are isomorphic, then the ML degrees of $V_d^{c_1}$ and $V_d^{c_2}$ coincide.
\end{conjecture}

\begin{outline*} \upshape
Section \ref{sec:EulerStratificationsOfQuadrics} relates the ML degrees of scaled toric varieties associated with second hypersimplices to the topological Euler characteristic of certain quadrics. Section \ref{sec:QuadricsOfRankThreeAndFour} studies scaling matrices of rank three and four and provides formulas for their associated ML degrees. Section \ref{sec:DeltaMatroidsFromSymmetricMatrices} explains how delta-matroids arise from symmetric matrices, which are then used in Section \ref{sec:SecondHypersimplexOfOrderFiveAndSix} to study the Euler stratification of second hypersimplices of order five and six. The article concludes with a discussion of future research directions in Section \ref{sec:Discussion}.
\end{outline*}

\noindent\textbf{Research Data.} The code associated with this work is publicly available on \texttt{Zenodo}~\cite{zenodo}.

\section{Euler Stratifications of Quadric Surfaces}
\label{sec:EulerStratificationsOfQuadrics}

Following Theorem~\ref{thm:ML-degree-topological-Euler-characteristic}, we study the Euler stratification of the scaled toric variety $V_d^c$ through its topological Euler characteristic. In this setting, computing the signed Euler characteristic of $V_d^c\setminus\cH$ reduces to determining the Euler characteristic of certain quadrics in the dense torus of $\CC\PP^{d-1}$. These quadrics are defined by the polynomial whose exponent vectors are precisely the lattice points of the second hypersimplex of order $d$. 

More specifically, for a given scaling $c\in(\CC^*)^{d(d-1)/2}$, we consider $$f=\theta_0(c_{01}\theta_1+\ldots+c_{0{d-1}}\theta_{d-1})+\sum_{1\le i<j\le d-1}c_{ij}\theta_i\theta_j \in\CC[\theta_0,\ldots,\theta_{d-1}].$$ For our further analysis, we write $f$ as $$f = \theta_0f_0+f_1, \qquad f_0\coloneqq c_{01}\theta_1+\ldots+c_{0d-1}\theta_{d-1}, \qquad f_1\coloneqq\sum_{1\le i<j\le d-1}c_{ij}\theta_i\theta_j.$$ 
Let $T\subseteq \CC\PP^{d-1}$ denote the dense torus of $\CC\PP^{d-1}$, i.e., $T\simeq (\CC^*)^{d-1} \subseteq \CC\PP^{d-1}$, and set $$\cV(f)\coloneqq\{ [\theta_0:\theta_1:\cdots:\theta_d]\in T \mid f(\theta_0,\theta_1,\ldots,\theta_d)=0 \}.$$ The ML degree of $V_d^c$ can then be expressed in terms of the Euler characteristic of $\cV(f)$. A similar description appears in \cite[Section 3]{TheEulerStratificationForP1P1Pn}, where the authors study Euler stratifications of three-way independence models, extending the results of \cite[Section~2]{MatroidStratificationOfMLDegreesOfIndependenceModels}.

\begin{proposition} \label{prop:ML-degree-f-polynomial}
For the scaled toric variety $V_d^c$, $$\mldeg(V_d^c) = (-1)^d\chi(\cV(f)) = (-1)^{d-1} \chi(\cV(f_0f_1)).$$
\end{proposition}

\begin{proof}
Let $\cV(f)^c$ denote the complement of $\cV(f)$. If $\psi^c$ is the parameterization associated with $\Delta_{d,2}$, its restriction to $\cV(f)^c$ induces an isomorphism $\cV(f)^c \simeq V_d^c \setminus\cH$. Hence, both varieties have the same Euler characteristic. Since $\dim (V_d^c)=d-1$ and $\chi(\cV(f)^c) = -\chi(\cV(f))$, $$\mldeg(V_d^c) = (-1)^{d-1}\chi(V_d^c\setminus\cH) = (-1)^{d-1}\chi(\cV(f)^c) = (-1)^d \chi(\cV(f)).$$ For further details on this correspondence, we refer the reader to \cite[Section 6.4]{EulerStratificationsOfHypersurfaceFamilies}. Moreover, by \cite[Theorem 2.2]{EulerDiscriminantOfComplementsOfHyperplanes}, we have $\chi(\cV(f))=-\chi(\cV(f_0f_1))$, and therefore \begin{equation*}\mldeg(V_d^c)=(-1)^{d-1}\chi(\cV(f_0f_1)). \qedhere \end{equation*}
\end{proof}

By inclusion-exclusion, we further decompose $\chi(\cV(f_0f_1))$ into three parts: \begin{equation*} \label{eq:inclusion-exclusion} \chi(\cV(f_0f_1)) = \chi(\cV(f_0)) + \chi(\cV(f_1)) - \chi(\cV(f_0)\cap\cV(f_1)).\end{equation*}

The first contribution comes from the hyperplane $\cV(f_0)$.

\begin{lemma} \label{lemma:Euler-characteristic-f0}
We have $\chi(\cV(f_0)) = (-1)^{d-1}$.
\end{lemma}

\begin{proof}
We first consider the hypersurface $f_0=0$ in the complex projective space, i.e., $$H\coloneqq\{[\theta_1:\cdots:\theta_{d-1}]\in\CC\PP^{d-2}\mid c_{01}\theta_1+\ldots+c_{0d-1}\theta_{d-1}=0 \}.$$ Since $H$ is a hyperplane in $\CC\PP^{d-2}$, we have $H\simeq\CC\PP^{d-3}$ and therefore $$\chi(H)=\chi(\CC\PP^{d-3})=d-2.$$ The variety $\cV(f_0)$ is obtained from $H$ by removing all points with zero coordinates.  
For $i\in[d-1]$, set $H_i\coloneqq H\cap\{ \theta_i=0\}$. Then $$\chi(\cV(f_0))= \chi\left(H\setminus\bigcup_{i=1}^{d-1}H_i\right)=\chi(H)-\chi\left(\bigcup_{i=1}^{d-1}H_i\right).$$ Each $H_i$ is a hyperplane in $H$, i.e., $H_i\simeq\CC\PP^{d-4}$, and therefore $\chi(H_i)=d-3$. More generally, for $k\in[d-3]$ and distinct indices $i_1,\ldots, i_k$, we have $H_{i_1}\cap\ldots\cap H_{i_k}\simeq\CC\PP^{d-3-k}$, so that $\chi(H_{i_1}\cap\ldots\cap H_{i_k})=d-2-k$. If $k\ge d-2$, the intersection is empty and its Euler characteristic is zero. By the inclusion-exclusion principle and a straightforward computation, we obtain \begin{align*}
\chi\left(\bigcup_{i=1}^{d-1}H_i\right) &= \sum_{k=1}^{d-1}(-1)^{k+1}\sum_{1\le i_1<\ldots< i_k\le d-1} \chi(H_{i_1}\cap\ldots\cap H_{i_k}) \\
&= \sum_{k=1}^{d-3}(-1)^{k+1}\binom{d-1}{k}(d-2-k) = (d-2)-(-1)^{d-1}.
\end{align*}
Consequently, \begin{equation*}\chi(\cV(f_0))=\chi(H)-\chi\left(\bigcup_{i=1}^{d-1}H_i\right)=(d-2)-[(d-2)-(-1)^{d-1}]=(-1)^{d-1}.\qedhere\end{equation*}
\end{proof}

The second term $\chi(\cV(f_1))$ corresponds to the second hypersimplex $\Delta_{d-1,2}$. Hence, knowing the ML degree of $V_{d-1}^c$ for the respective scaling yields $\chi(\cV(f_1))$ via Proposition \ref{prop:ML-degree-f-polynomial}.

The last contribution arises from $\cV(f_0)\cap\cV(f_1)$, whose geometric interpretation is as follows. The variety $\cV(f)$ is a quadric surface in the torus $T$, with representing matrix $\frac{1}{2}C$: $$f = \theta_0 f_0 + f_1 = (\theta_0,\ldots,\theta_{d-1}) \frac{1}{2} C (\theta_0,\ldots,\theta_{d-1})^T.$$

\begin{lemma} \label{lemma:intersection-rank}
The intersection $\cV(f_0)\cap\cV(f_1)$ is a quadric whose representing matrix has rank $$\rank(C)-2.$$
\end{lemma}

\begin{proof}
For a fixed scaling $c\in(\CC^*)^{d(d-1)/2}$, we write the scaling matrix $C\in\CC^{d\times d}$ in block~form $$C = \begin{pmatrix} X & Y \\ Y^T & Z \end{pmatrix}, \qquad \textup{where } X = \begin{pmatrix} 0 & c_{01} \\ c_{01} & 0 \end{pmatrix}.$$ Since $c_{01}\neq0$, the matrix $X$ is invertible. Now consider the quadric $$f(\theta)=\theta^TC\theta,  \qquad \theta = (\theta_0, \ldots, \theta_{d-1}) \in \CC\PP^{d-1}.$$ Writing $\theta = (x_1,x_2) \in \CC\PP^1\times\CC\PP^{d-3}$, and since $C$ is symmetric, we obtain $$f(\theta) = x_1^T X x_1 + 2x_1^T Y x_2 + x_2^T Z x_2.$$ The linear system $X x_1+Yx_2=0$ consists of two equations, one for each coordinate of $x_1 = (\theta_0,\theta_1)^T$. Since $X$ is invertible, solving this system yields $x_1 = -X^{-1}Yx_2$. Since $f_1$ does not depend on $\theta_0$, substituting $x_1 = -X^{-1}Yx_2$ into $f(\theta)$ restricts $f_1=0$ to the hyperplane $\cV(f_0)$. This results in a quadric in $x_2$ of the form $$f(-X^{-1}Yx_2, x_2) = x_2^T (Z-Y^TX^{-1}Y)x_2.$$ Thus, $\cV(f_0)\cap\cV(f_1)$ is precisely the quadric defined by the Schur complement $Z-Y^TX^{-1}Y$. By the Guttman rank additivity formula \cite[page 14]{TheSchurComplementAndItsApplications}, since $\rank(X)=2$, we conclude \begin{equation*}\rank(C) = \rank(X) + \rank(Z-Y^TX^{-1}Y) = 2 + \rank(Z-Y^TX^{-1}Y). \qedhere \end{equation*}
\end{proof}

This yields a complete geometric interpretation of the last term in the inclusion-exclusion formula. We illustrate our methodology for the second hypersimplex of order four. Using an alternative approach, its Euler stratification was first described in \cite[Corollary~3.8]{MatroidStratificationOfMLDegreesOfIndependenceModels}.

\begin{example} \label{example:SecondHypersimplexOfOrderFour} \upshape
The principal $A$-determinant associated with $\Delta_{4,2}$ is $\det(C)$, where $$C = \begin{pmatrix} 0 & c_{01} & c_{02} & c_{03} \\ c_{01} & 0 & c_{12} & c_{13} \\ c_{02} & c_{12} & 0 & c_{23} \\ c_{03} & c_{13} & c_{23} & 0 \end{pmatrix}.$$ The defining polynomial is $$f = \theta_0 (\underbrace{c_{01} \theta_1 +  c_{02} \theta_2 + c_{03} \theta_3}_{=:f_0}) + \underbrace{c_{12} \theta_1 \theta_2 + c_{13} \theta_1 \theta_3 + c_{23} \theta_2 \theta_3}_{=:f_1}.$$ By Lemma \ref{lemma:Euler-characteristic-f0}, we have $\chi(\cV(f_0)) = -1$. 

The polynomial $f_1$ defines a smooth conic in $\CC\PP^2$, i.e., a genus-zero curve isomorphic to $\CC\PP^1$. Passing to the dense torus removes the points $[1:0:0]$, $[0:1:0]$, $[0:0:1]$. Hence, $$\chi(\cV(f_1)) = 2-3=-1.$$

By Lemma~\ref{lemma:intersection-rank}, the intersection $\cV(f_0)\cap\cV(f_1)$ is a quadric in $\CC\PP^1$ consisting of two points, counted with multiplicities. Substituting $\theta_3=-1/c_{03}(c_{01}\theta_1+c_{02}\theta_2)$ into $f_1$ yields $$c_{01}c_{13}\theta_1^2+(c_{02}c_{13}+c_{01}c_{23}-c_{03}c_{12})\theta_1\theta_2+c_{02}c_{23}\theta_2^2=0.$$ Its discriminant equals the principal $A$-determinant. Therefore, $$\chi(\cV(f_0)\cap\cV(f_1)) = \begin{cases} 2, & \textup{if } \det(C)\neq0, \\ 1, & \textup{otherwise.} \end{cases}$$

Hence, for generic scalings, \begin{align*}
\mldeg(V_4^c) &= (-1)^3 \left( \chi(\cV(f_0)) + \chi(\cV(f_1)) - \chi(\cV(f_0)\cap\cV(f_1)) \right) \\ &= (-1)^3 (-1+(-1)-2) = 4.
\end{align*}
For non-generic scalings, the ML degree drops accordingly by one.
\end{example}

In general, the Euler characteristic of $\cV(f)$ depends, inter alia, on the rank of its representing matrix. We study this dependence in more detail in the following section.

\section{Quadrics of Rank Three and Four}
\label{sec:QuadricsOfRankThreeAndFour}

We first focus on the scaled toric varieties $V_d^c$ for scaling matrices $C\in\CC^{d\times d}$ of rank three. Equivalently, the corresponding polynomial $f$ is a quadric of rank three. 
This corresponds to scalings for which all factors of the principal $A$-determinant vanish \cite[Theorem~1.1]{ThePrincipalRankCharacteristicSequenceOverVariousFields}. In \cite[Conjecture 3.9]{MatroidStratificationOfMLDegreesOfIndependenceModels}, the authors conjectured that the minimum ML degree of $V_d^c$ over all possible scalings is given by a binomial coefficient. The following theorem shows that this particular ML degree is indeed attained by scaling matrices of rank three.

\begin{theorem} \label{thm:ML-degree-rank-three}
Let $c\in(\CC^*)^{d(d-1)/2}$ be a scaling, for which $C\in\CC^{d\times d}$ has rank three. Then $$\mldeg(V_d^c)=\binom{d-1}{2}.$$
\end{theorem}

\begin{proof}
We prove this by induction on $d$. For $d=4$, the statement follows from Example~\ref{example:SecondHypersimplexOfOrderFour}. Now consider $V_{d+1}^c$ and assume that $c\in(\CC^*)^{(d+1)d/2}$ is a scaling for which all factors of the principal $A$-determinant vanish. By Section \ref{sec:EulerStratificationsOfQuadrics}, we compute the ML degree of $V_{d+1}^c$ via $$\mldeg(V_{d+1}^c) = (-1)^d ( \chi(\cV(f_0)) + \chi(\cV(f_1)) - \chi(\cV(f_0)\cap\cV(f_1))).$$ 

According to Lemma~\ref{lemma:Euler-characteristic-f0}, we have $\chi(\cV(f_0)) = (-1)^d$. Moreover, Proposition \ref{prop:ML-degree-f-polynomial} together with the inductive hypothesis yields $$\chi(\cV(f_1)) = (-1)^d \binom{d-1}{2}.$$

It remains to compute $\cV(f_0)\cap\cV(f_1)$. By Lemma \ref{lemma:intersection-rank}, this is a quadric in $\CC\PP^{d-2}$ of rank one. Hence, $\cV(f_0)\cap\cV(f_1)$ coincides with a hyperplane $H\simeq\CC\PP^{d-3}$, and $\chi(H)=\chi(\CC\PP^{d-3})=d-2$. Passing to the dense torus removes the coordinate hyperplane sections of $H$. Arguing exactly as in the proof of Lemma \ref{lemma:Euler-characteristic-f0}, inclusion-exclusion yields $$\chi(\cV(f_0)\cap\cV(f_1)) = (-1)^{d+1}(d-2).$$ Consequently, 
\begin{align*}
\mldeg(V_{d+1}^c) &= (-1)^d \left((-1)^d+(-1)^d\binom{d-1}{2}-(-1)^{d+1}(d-2)\right) \\ 
&=  \binom{d-1}{2}+\binom{d-1}{1}=\binom{d}{2}. \qedhere
\end{align*}
\end{proof}

The following lemma quarantees the existence of a scaling for which the polynomial associated with $\Delta_{d,2}$ is a quadric of rank three. Consequently, the ML degree determined in Theorem~\ref{thm:ML-degree-rank-three} is indeed realizable for suitably scaled toric varieties defined by the second hypersimplex of arbitrary order. The proof provides an explicit construction of such a scaling.

\begin{proposition} \label{prop:rank-three-scaling}
There exists a scaling $c\in(\CC^*)^{d(d-1)/2}$ such that $C\in\CC^{d\times d}$ has rank three.
\end{proposition}

\begin{proof}
We construct a symmetric $d\times d$ matrix of rank three with zero diagonal and nonzero off-diagonal entries. Let $X\in\RR^{d\times3}$ be a matrix with entries $x_{i1}=1$, $x_{i2}=i$ and $x_{i3}=\sqrt{1+i^2}$ for all $i\in[d]$. Then, for every $i\in[d]$, $$x_{i1}^2+x_{i2}^2-x_{i3}^2 = 1+i^2-(1+i^2) = 0.$$ Moreover, for $i\neq j$, $$x_{i1}x_{j1}+x_{i2}x_{j2}-x_{i3}x_{j3} = 1+ij-\sqrt{1+i^2}\sqrt{1+j^2} \neq 0.$$ Here, we note that the equality $1+ij = \sqrt{1+i^2}\sqrt{1+j^2}$ holds if and only if $i=j$. 

To verify that $\textup{rank}(X)=3$, we consider the following equation componentwise: $$\lambda_1 (1, \ldots, 1)^T + \lambda_2 (1, \ldots, d)^T + \lambda_3 (\sqrt{1+1^2}, \ldots, \sqrt{1+d^2})^T = 0.$$ If $\lambda_3\neq0$, this would imply that $\sqrt{1+i^2}$ is an affine function of $i$, which is impossible. Hence $\lambda_3=0$, and then $\lambda_1+\lambda_2\cdot i=0$ for all $i\in[d]$ implies $\lambda_1=\lambda_2=0$. Therefore, $\textup{rank}(X)=3$.

Now let $J=\textup{diag}(1,1,-1)$, and define $C\coloneqq XJX^T$. Then $C$ is symmetric, and $$\textup{rank}(C)=\textup{rank}(XJX^T)\le \min\{\textup{rank}(X), \textup{rank}(J), \textup{rank}(X^T) \}=3.$$ Consequently, all principal $k$-minors of $C$ vanish for $k\ge4$. The diagonal entries of $C$ are $$C_{ii} = x_{i1}^2+x_{i2}^2-x_{i3}^2=0,$$ while for $i\neq j$, $$C_{ij} = x_{i1}x_{j1}+x_{i2}x_{j2}-x_{i3}x_{j3}\neq0.$$ Thus, $C$ is a symmetric matrix of rank three with zero diagonal and non-zero off-diagonal, i.e., its entries define a scaling for which all factors of the principal $A$-determinant vanish.
\end{proof}

Proposition~\ref{prop:rank-three-scaling} combined with Theorem~\ref{thm:ML-degree-rank-three} completes the proof of Theorem~\ref{thm:main-rank-three}. For the investigations in Section \ref{sec:SecondHypersimplexOfOrderFiveAndSix}, we derive an analogous formula for scaling matrices of rank four.

\begin{theorem} \label{thm:rank-four-ML-degree}
Let $C\in\CC^{d\times d}$ be a scaling matrix of rank four. Then $$\mldeg(V_d^c) = \binom{d}{3} + \sum_{i=4}^{d-1} (-1)^{i-1} C_{3,i},$$ where $C_{3,i}$ denotes the number of rank-three principal $i$-submatrices of $C$.
\end{theorem}

\begin{proof}
We prove this by induction on $d$. For $d=4$, the sum is empty, so the claim follows immediately from Example \ref{example:SecondHypersimplexOfOrderFour}. Now assume that the statement holds for all rank-four scaling matrices of size at most $d-1$, and let $C\in\CC^{d\times d}$ be a scaling matrix of rank four. 

By Lemma \ref{lemma:Euler-characteristic-f0}, we have $\chi(\cV(f_0)) = (-1)^{d-1}$. 

Let $X = \{ 2, \ldots, d \}$, and consider the principal submatrix $C[X]$ with associated scaling~$c'$. If $C[X]$ has rank four, then by the inductive hypothesis and Proposition \ref{prop:ML-degree-f-polynomial}, $$\chi(\cV(f_1)) = (-1)^{d-1} \mldeg(V_{d-1}^{c'}) = (-1)^{d-1} \left[ \binom{d-1}{3} + \sum_{i=4}^{d-1} (-1)^{i-1} C_{3,i}^1 \right],$$ where $C_{3,i}^1$ denotes the number of rank-three principal $i$-submatrices of $C[X]$, with $C_{3,d-1}^1=0$.

If instead $C[X]$ has rank three, then by Theorem \ref{thm:ML-degree-rank-three} and Proposition \ref{prop:ML-degree-f-polynomial}, $$\chi(\cV(f_1)) = (-1)^{d-1} \mldeg(V_{d-1}^{c'}) = (-1)^{d-1} \binom{d-1}{2}.$$ In this case, every principal submatrix of size $i\ge 4$ has rank three, i.e., $C_{3,i}^1 = \binom{d-1}{i}$. Since $$\binom{d-2}{2} = \binom{d-1}{3} + \sum_{i=4}^{d-1} (-1)^{i-1} \binom{d-1}{i},$$ we again obtain $$\chi(\cV(f_1)) = (-1)^{d-1} \left[ \binom{d-1}{3} + \sum_{i=4}^{d-1} (-1)^{i-1} C_{3,i}^1 \right].$$ 

Next consider $\cV(f_0)\cap\cV(f_1)$. Since $C$ has rank four, this is a rank-two quadric in $\CC\PP^{d-3}$, whose Euler characteristic equals $d-2$. Passing to the dense torus yields $$\chi(\cV(f_0)\cap\cV(f_1)) = (d-2) + \sum_{i=1}^{d-4} (-1)^i (d-i-2) \left( \binom{d}{d-i} - \binom{d-1}{d-i} \right) + \sum_{i=1}^{d-4} (-1)^{i-1} C_{3,d-i}^2,$$ where $C_{3,d-i}^2 \coloneqq C_{3,d-i} - C_{3,d-i}^1$.

We now compute $$\mldeg(V_d^c) = (-1)^{d-1} \left(\chi(\cV(f_0)) + \chi(\cV(f_1)) - \chi(\cV(f_0)\cap\cV(f_1))\right).$$ 

The contribution of the terms involving rank-three principal submatrices is $$(-1)^{d-1} \sum_{i=4}^{d-1} (-1)^{i-1} C_{3,i}^1 -  \sum_{i=1}^{d-4} (-1)^{i-1} C_{3,d-i}^2.$$ Reindexing the second sum via $j=d-i$ and using $(-1)^{d+i-2} + (-1)^{d-i-1} = 0$ together with $C_{3,i}^1 = C_{3,i}-C_{3,i}^2$, this simplifies to $\sum_{i=4}^{d-1} (-1)^{d+i-2} C_{3,i}$.

For the constant term, we obtain $$(-1)^{d-1} + (-1)^{d-1} \binom{d-1}{3} - (d-2) - \sum_{i=1}^{d-4} (-1)^i (d-i-2) \binom{d-1}{i}.$$ Using the identities $\sum_{i=0}^n (-1)^i \binom{n}{i} = 0$ and $\sum_{i=0}^n (-1)^i i \binom{n}{i} = 0$, this simplifies to $(-1)^{d-1} \binom{d}{3}$.

Hence, combining the previous computations yields \begin{equation*}\mldeg(V_d^c) = (-1)^{d-1} \left( (-1)^{d-1} \binom{d}{3} + \sum_{i=4}^{d-1} (-1)^{d+i-2} C_{3,i} \right) = \binom{d}{3} + \sum_{i=4}^{d-1} (-1)^{i-1} C_{3,i}. \qedhere \end{equation*}
\end{proof}

\section{Delta-matroids from Symmetric Matrices}
\label{sec:DeltaMatroidsFromSymmetricMatrices}

Let $C$ be a symmetric matrix over a field, whose rows and columns are indexed by $[d]$. For $X\subseteq[d]$, let $C[X]$ denote the principal submatrix indexed by $X$. We define $$\cF\coloneqq\{ X \subseteq [d] \mid \det C[X] \neq 0 \}.$$ Note that $\cF$ is nonempty since $\det C[\emptyset] = 1$. For feasible sets $X,Y\in\cF$, their symmetric difference is defined as $$X\Delta Y \coloneqq (X\setminus Y)\cup (Y\setminus X).$$ According to \cite[Theorem 4.1]{RepresentabilityOfDeltaMatroids}, the pair $([d],\cF)$ forms a delta-matroid, that is, it satisfies the symmetric exchange axiom.

\begin{symexchange*}
    For every $X,Y\in\cF$ and every $x\in X\Delta Y$ there exists $y\in X\Delta Y$ such that $X\Delta\{x,y\}\in\cF$.
\end{symexchange*}

We outline a proof based on \emph{principal pivot transforms} \cite[Section~2.4]{GammaGraphicDeltaMatroidsAndTheirApplications}.
Assume $C$ is written in block form $$C = \begin{pmatrix} C[X] & B \\ B^T & C[[d]\setminus X] \end{pmatrix}.$$ If $X\in\cF$, i.e., $C[X]$ is invertible, the principal pivot transform of $C$ on $X$ is defined as $$C\star X = \begin{pmatrix} C[X]^{-1} & C[X]^{-1} B \\ -B^T C[X]^{-1} & C[[d]\setminus X]-B^TC[X]^{-1}B \end{pmatrix}.$$

\begin{lemma} \label{lemma:transform-determinant}
If $X\in\mathcal{F}$, then for all $Y\subseteq [d]$, $\det C\star X[Y] = 0$ if and only if $\det C[X\Delta Y] = 0$.
\end{lemma}

\begin{proof}
This follows from \cite[Property 4.3.1]{RepresentabilityOfDeltaMatroids}, see also \cite[\S 4 Theorem~4.1.2]{TheLinearComplementarityProblem} and \cite[Theorem~2.1]{AGeneralizationOfTuttesCharacterizationOfTotallyUnimodularMatrices}.
\end{proof}

The above lemma is the main tool for the proof of Bouchet's result \cite{RepresentabilityOfDeltaMatroids}.

\begin{theorem} \label{thm:symmetric-matrix-delta-matroid}
If $C$ is a symmetric matrix of size $d$, then $([d],\mathcal{F})$ is a delta-matroid.
\end{theorem}

\begin{proof}
We show that $\cF$ satisfies the symmetric exchange axiom. 

Let $X,Y\in\mathcal{F}$ and let $x\in X\Delta Y$. By Lemma \ref{lemma:transform-determinant}, we have $\det C\star X[X\Delta Y]\neq 0$, since $$X\Delta(X\Delta Y) = (X\Delta X)\Delta Y = \emptyset\Delta Y = Y \in\mathcal{F}.$$ In the following, we show that there exists $y\in X\Delta Y$ such that $\det C\star X[\{x,y\}]\neq 0$. Then, by Lemma \ref{lemma:transform-determinant}, it follows that $X\Delta\{x,y\}\in\mathcal{F}$. We distinguish between two cases. \medskip

\textbf{\textit{Case 1.}} If $\det C\star X[\{x\}]\neq0$, then $X\Delta\{x\}\in\mathcal{F}$ by Lemma \ref{lemma:transform-determinant}. Taking $y=x$, we have $$X\Delta\{x,y\}=X\Delta\{x\}\in\mathcal{F}.$$ 

\textbf{\textit{Case 2.}} If $\det C\star X[\{x\}]=0$, let $c_{xy}'$ denote the $(x,y)$-entry of $C\star X$. By assumption, $c_{xx}'=0$. Since $C\star X[X\Delta Y]$ is non-singular and $x\in X\Delta Y$, the row of $C\star X[X\Delta Y]$ indexed by $x$ cannot be a zero row. Therefore, there exists $y\in X\Delta Y\setminus\{x\}$ such that $c_{xy}'\neq0$. By \cite[Theorem 7.4]{DeltaMatroidsForGraphTheorists}, the principal pivot transform of a symmetric matrix is symmetric.~Hence, \begin{equation*} \det C\star X[\{x,y\}] = \det \begin{pmatrix} c_{xx}' & c_{xy}' \\ c_{xy}' & c_{yy}'\end{pmatrix} = \det \begin{pmatrix} 0 & c_{xy}' \\ c_{xy}' & c_{yy}'\end{pmatrix} = - (c_{xy}')^2 \neq 0. \qedhere \end{equation*} 
\end{proof}

We will use the contrapositive of Theorem~\ref{thm:symmetric-matrix-delta-matroid}, that is, if $\cF\subseteq 2^{[d]}$ does not form a delta-matroid, then it cannot arise from the non-singular principal submatrices of a symmetric matrix. We particularly apply this principle to rule out certain scaling matrices.

\section{Second Hypersimplices of Order Five and Six}
\label{sec:SecondHypersimplexOfOrderFiveAndSix}

We first investigate the second hypersimplex of order five. For generic scalings $c\in(\CC^*)^{10}$, we have $\mldeg(V_5^c) = 11$. However, according to Section~\ref{sec:QuadricsOfRankThreeAndFour}, the ML degree can be reduced to six by choosing a suitable scaling. We now show that six is indeed minimal.

\begin{proposition} \label{prop:stratification-second-hypersimplex-of-order-five}
The minimum ML degree of $V_5^c$ that can be achieved by choosing a suitable scaling $c\in(\CC^*)^{10}$ is six.
\end{proposition}

\begin{proof}
Let $C\in\CC^{5\times 5}$ be a scaling matrix, and let $k$ denote the number of non-vanishing principal $4$-minors of $C$. Write $k=k_1+k_2$ where $k_1=1$ if $\det C[\{2,3,4,5\}]\neq0$ and $k_1=0$ otherwise. We distinguish the following three cases according to the rank of $C$.
\begin{itemize}
\item[1.] If $\rank(C)=3$, then $\mldeg(V_5^c)=6$ by Theorem~\ref{thm:ML-degree-rank-three}.
\item[2.] If $\rank(C)=4$, then Theorem \ref{thm:rank-four-ML-degree} yields $\mldeg(V_5^c) = 5+k$, noting that $C_{3,4} = 5-k$. Since $k>0$ whenever $\textup{rank}(C)=4$, the ML degree is at least six in this case.
\item[3.] If $\textup{rank}(C)=5$, then $\cV(f_0)\cap\cV(f_1)$ is a smooth quadric in $\CC\PP^2$, whose Euler characteristic equals two. Passing to the dense torus requires removing points with zero coordinates. Setting $\theta_i=0$ for $i\in[4]$ yields the quadric corresponding to the second hypersimplex of order four. By Example~\ref{example:SecondHypersimplexOfOrderFour}, its Euler characteristic equals one or two depending on the vanishing of the corresponding principal $4$-minor. Hence, $$\chi(\cV(f_0)\cup\cV(f_1)) = 2 - (2k_2+(4-k_2)) = -2-k_2.$$ Since $\chi(\cV(f_0))=1$ and $\chi(\cV(f_1))= 3+k_1$ by Lemma~\ref{lemma:Euler-characteristic-f0} and Example~\ref{example:SecondHypersimplexOfOrderFour}, we obtain $$\mldeg(V_5^c) = (-1)^4 (1+(3+k_1)-(-2-k_2)) = 6+k_1+k_2 = 6+k.$$
\end{itemize}
Therefore, the minimum ML degree achievable by a suitable scaling is six.
\end{proof}

The proof of Proposition~\ref{prop:stratification-second-hypersimplex-of-order-five} shows that the ML degree of $V_5^c$ is determined by the pattern of vanishing factors of the principal $A$-determinant. We now use delta-matroids to investigate which of these vanishing patterns can actually occur.

\begin{example} \label{ex:delta-matroids-size-five}
A scaling matrix $C\in\CC^{5\times 5}$ has zero diagonal, and non-vanishing principal $2$- and $3$-minors. Therefore, we consider families $\cF\subseteq 2^{[5]}$ that contain the empty set as well as all $2$- and $3$-subsets of $[5]$, but no singleton. Depending on whether $[5]\in\cF$ or $[5]\notin\cF$, we determine the possible combinations of $4$-subsets that may be included so that $([5],\cF)$ forms a delta-matroid. To determine the permissible possibilities in both cases, we performed brute-force computations in \texttt{Macaulay2} \cite{Macaulay2}. The results are shown in Table~\ref{table:number-of-delta-matroids-size-five}, and can also be verified by hand by checking whether the symmetric exchange axiom is satisfied.
\end{example}

\begin{table}[htb] \centering
\begin{tabular}{|| l || C{1cm} C{1cm} C{1cm} C{1cm} C{1cm} C{1cm} ||} 
 \hline
 & \multicolumn{6}{|c|}{Number of $4$-subsets included} \\
 \hline
 & 0 & 1 & 2 & 3 & 4 & 5 \\
 \hline\hline
 $[5]\notin\cF$ & 1 & 0 & 0 & 0 & 5 & 1 \\
 \hline
 $[5]\in\cF$& 1 & 5 & 10 & 10 & 5 & 1\\ 
 \hline
\end{tabular}

\caption{Number of delta-matroids arising from $\cF \subseteq 2^{[5]}$ as considered in Example~\ref{ex:delta-matroids-size-five}.} \label{table:number-of-delta-matroids-size-five}
\end{table}

Table~\ref{table:number-of-delta-matroids-size-five} implies that a rank-four scaling matrix $C\in\CC^{5\times5}$ must contain at least four non-vanishing principal $4$-minors. Referring back to the proof of Proposition~\ref{prop:stratification-second-hypersimplex-of-order-five}, we conclude that the ML degree is at least nine in this case. Moreover, we verified that every delta-matroid listed in
Table~\ref{table:number-of-delta-matroids-size-five}
is indeed representable by a scaling matrix associated with $\Delta_{5,2}$.

\bigskip

We now turn to the second hypersimplex of order six. For generic scalings $c\in(\CC^*)^{15}$, the ML degree of $V_6^c$ equals $26$. We will show that the minimum achievable ML degree is $10$, thereby proving Theorem~\ref{thm:main-minimum} for $d=6$. In contrast to the $d=5$ case, the proof of the minimum ML degree requires a more detailed analysis of the associated delta-matroids. 

Analogously to the case of ground set $[5]$, we determined the cases where no delta-matroid exists via brute-force computations. According to whether $[6]\notin\cF$ or $[6]\in\cF$, the corresponding results are shown in Table~\ref{table:number-of-delta-matroids-size-six-without-ground-set} and Table~\ref{table:number-of-delta-matroids-size-six-with-ground-set}, respectively.

\begin{table}[tbh] \centering 
\begin{tabular}{|| c || C{1cm} C{1cm} C{1cm} C{1cm} C{1cm} C{1cm} C{1cm} ||} 
 \hline
 $[6]\notin\cF$ & \multicolumn{7}{|c|}{Number of $5$-subsets included} \\
 \hline
 Number of $4$-subsets included & 0 & 1 & 2 & 3 & 4 & 5 & 6 \\[0.5ex] 
 \hline\hline
 0 & 1 & 0 & 0 & 0 & 15 & 6 & 1 \\ 
 \hline
 1 & 0 & 0 & 0 & 0 & 90 & 60 & 15 \\ 
 \hline
 2 & 0 & 0 & 0 & 0 & 225 & 270 & 105 \\ 
 \hline
 3 & 0 & 0 & 0 & 0 & 300 & 720 & 455 \\ 
 \hline
 4 & 0 & 0 & 0 & 0 & 255 & 1260 & 1365 \\ 
 \hline
 5 & 0 & 0 & 0 & 0 & 270 & 1518 & 3003 \\ 
 \hline
 6 & 0 & 0 & 0 & 0 & 465 & 1320 & 5005 \\ 
 \hline
 7 & 0 & 0 & 0 & 0 & 600 & 990 & 6435 \\ 
 \hline
 8 & 0 & 0 & 0 & 0 & 465 & 990 & 6435 \\ 
 \hline
 9 & 0 & 0 & 0 & 0 & 270 & 1320 & 5005 \\ 
 \hline
 10 & 6 & 0 & 0 & 0 & 255 & 1518 & 3003 \\ 
 \hline
 11 & 0 & 0 & 0 & 0 & 300 & 1260 & 1365 \\ 
 \hline
 12 & 15 & 0 & 0 & 0 & 225 & 720 & 455 \\ 
 \hline
 13 & 45 & 0 & 0 & 0 & 90 & 270 & 105 \\ 
 \hline
 14 & 15 & 0 & 0 & 0 & 15 & 60 & 15 \\ 
 \hline
 15 & 1 & 0 & 0 & 0 & 0 & 6 & 1 \\ 
 \hline 
\end{tabular}

\caption{Number of delta-matroids arising from $\cF\subseteq2^{[6]}$ where $[6]\notin\cF$ and $\cF$ contains the empty set, all $2$- and $3$-subsets, but no subset of cardinality one.} \label{table:number-of-delta-matroids-size-six-without-ground-set}
\end{table}

\begin{table}[tbh] \centering 
\begin{tabular}{|| c || C{1cm} C{1cm} C{1cm} C{1cm} C{1cm} C{1cm} C{1cm} ||} 
 \hline
 $[6]\in\cF$ & \multicolumn{7}{|c|}{Number of $5$-subsets included} \\
 \hline
 Number of $4$-subsets included & 0 & 1 & 2 & 3 & 4 & 5 & 6 \\[0.5ex] 
 \hline\hline
 0 & 0 & 0 & 0 & 0 & 0 & 0 & 1 \\ 
 \hline
 1 & 0 & 0 & 0 & 0 & 0 & 0 & 15 \\ 
 \hline
 2 & 0 & 0 & 0 & 0 & 0 & 0 & 105 \\ 
 \hline
 3 & 0 & 0 & 0 & 0 & 0 & 0 & 455 \\ 
 \hline
 4 & 0 & 0 & 0 & 0 & 0 & 30 & 1365 \\ 
 \hline
 5 & 0 & 0 & 0 & 0 & 0 & 306 & 3003 \\ 
 \hline
 6 & 0 & 0 & 0 & 0 & 0 & 1410 & 5005 \\ 
 \hline
 7 & 0 & 0 & 0 & 0 & 240 & 3870 & 6435 \\ 
 \hline
 8 & 0 & 0 & 0 & 0 & 1560 & 7020 & 6435 \\ 
 \hline
 9 & 0 & 0 & 0 & 540 & 4335 & 8820 & 5005 \\ 
 \hline
 10 & 0 & 6 & 240 & 2160 & 6690 & 7812 & 3003 \\ 
 \hline
 11 & 0 & 30 & 720 & 3420 & 6225 & 4860 & 1365 \\ 
 \hline
 12 & 0 & 60 & 840 & 2720 & 3540 & 2070 & 455 \\ 
 \hline
 13 & 0 & 60 & 480 & 1140 & 1185 & 570 & 105 \\ 
 \hline
 14 & 0 & 30 & 135 & 240 & 210 & 90 & 15 \\ 
 \hline
 15 & 1 & 6 & 15 & 20 & 15 & 6 & 1 \\ 
 \hline 
\end{tabular}

\caption{Number of delta-matroids arising from $\cF\subseteq2^{[6]}$ with $[6]\in\cF$, where $\cF$ contains the empty set, all $2$- and $3$-subsets, but no subset of cardinality one.} \label{table:number-of-delta-matroids-size-six-with-ground-set}
\end{table}

In the remainder of this section, let $C\in\CC^{6\times6}$ denote the scaling matrix associated with a scaling $c\in(\CC^*)^{15}$. We study the ML degree of $V_6^c$ depending on the rank of $C$.

\begin{notation*}
For a scaling matrix $C\in\CC^{6\times6}$, let $k$ and $l$ denote the number of non-vanishing principal $4$- and $5$-minors, respectively. Furthermore, let $r$ denote the number of principal $5$-submatrices of rank three. 
\end{notation*}

\begin{proposition} \label{prop:ML-degree-order-6-rank-4}
If $\rank(C)=4$, then $\mldeg(V_6^c) = 5+k+r$.
\end{proposition}

\begin{proof}
This follows from Theorem \ref{thm:rank-four-ML-degree} using the fact that $C_{3,4}=15-k$ and $C_{3,5}=r$.
\end{proof}

We immediately obtain the following corollary.

\begin{corollary}
If $\rank(C)=4$, then the ML degree of $V_6^c$ is at least 16.
\end{corollary}

\begin{proof}
This follows from Proposition~\ref{prop:ML-degree-order-6-rank-4}, together with Table~\ref{table:number-of-delta-matroids-size-six-without-ground-set}. If the number of non-vanishing principal $4$-minors is $k=10$, a direct computation of the six delta-matroids shows that exactly one principal $5$-submatrix has rank three ($r=1$). Otherwise, the scaling matrix has at least twelve non-vanishing principal $4$-minors, so that the ML degree is at least 17. 
\end{proof}

We continue with scaling matrices $C\in\CC^{6\times6}$ of rank five.

\begin{table}[tbh] \centering 
\begin{tabular}{|| c || C{1cm} C{1cm} C{1cm} C{1cm} || c ||} 
 \hline
 Case & $k_1$ & $k_2$ & $l_1$ & $l_2$ & ML degree \\[0.5ex] 
 \hline\hline
 \multicolumn{6}{|c|}{all principal $4$-minors vanish} \\
 \hline\hline
 1 & $=0$ & $=0$ & $\ge0$ & $\ge0$ & $10$ \\ 
 \hline\hline
 \multicolumn{6}{|c|}{all principal $4$-minors do not vanish} \\
 \hline\hline
 2 & $=5$ & $=10$ & $\ge0$ & $\ge0$ & $4+k+l(+r)$ \\
 \hline\hline
 \multicolumn{6}{|c|}{some principal $4$-minors vanish} \\
 \hline\hline
 3.1 & $=0$ & $>0$ & $=0$ & $>0$ & $4+k+l+r$ \\
 \hline
 3.2 & $>0$ & $>0$ & $=0$ & $>0$ & $4+k+l+r$ \\
 \hline
 3.3 & $=0$ & $>0$ & $=1$ & $>0$ & $4+k+l+r$ \\
 \hline
 3.4 & $>0$ & $>0$ & $=1$ & $>0$ & $4+k+l+r$ \\
 \hline
 3.5 & $>0$ & $=0$ & $=0$ & $>0$ & $4+k+l+r$ \\
 \hline
 3.6 & $>0$ & $=0$ & $=1$ & $>0$ & $4+k+l+r$\\
 \hline 
\end{tabular}

\caption{ML degrees of $V_6^c$ for scalings $c\in(\CC^*)^{15}$ with $\rank(C)=5$.} \label{table:ML-degrees-order-6-rank-5}
\end{table}

\begin{proposition} \label{prop:ML-degree-order-6-rank-5}
If $\rank(C)=5$, then $\mldeg(V_6^c)=4+k+l+r$. In particular, the ML degree is at least ten.
\end{proposition}

\begin{proof}
Let $C\in\CC^{6\times6}$ be a scaling matrix of rank five. Then $\det C =0$ and there is at least one non-vanishing principal $5$-minor. For $X = \{2,3,4,5,6 \}$, let $k_1$ denote the number of non-vanishing principal $4$-minors of $C[X]$, and set $k_2\coloneqq k-k_1$. Similarly, let $l_2\coloneqq l-l_1$ and $r_2\coloneqq r-r_1$, where $$l_1 \coloneqq \begin{cases} 1, & \textup{if } \det C[X]\neq0, \\ 0, & \textup{if } \det C[X]=0,\end{cases} \qquad r_1 \coloneqq \begin{cases} 1, & \textup{if } \rank(C[X]) = 3, \\ 0, & \textup{otherwise.} \end{cases}$$

The ML degree depends on the pattern of vanishing principal $4$- and $5$-minors as well as on the number of principal $5$-submatrices of rank three. All possible cases are listed in Table~\ref{table:ML-degrees-order-6-rank-5}. We exclude the cases where $l_1=l_2=0$, since this contradicts $\rank(C)=5$, and the cases where $l_1=1$ and $l_2=0$, since these are impossible by Table \ref{table:number-of-delta-matroids-size-six-without-ground-set}. The ML degrees in Table~\ref{table:ML-degrees-order-6-rank-5} were determined by the procedure described in Section~\ref{sec:EulerStratificationsOfQuadrics}. 

We illustrates this procedure for Case 3.6 in Table~\ref{table:ML-degrees-order-6-rank-5}. Since $l_1=1$, we have $\rank C[X]=5$, and therefore $\chi(\cV(f_1)) = - (6+k_1)$ by the proof of Proposition~\ref{prop:stratification-second-hypersimplex-of-order-five}. Furthermore, by Lemma~\ref{lemma:Euler-characteristic-f0}, $\chi(\cV(f_0)) = (-1)^5$. The intersection $\cV(f_0)\cap\cV(f_1)$ corresponds to a rank-three quadric in $\CC\PP^3$, which has Euler characteristic three. Passing to the dense torus yields $$\chi(\cV(f_0)\cap\cV(f_1)) = 3 - 5\cdot 2 + 10 = 3.$$ The correction terms arise as follows. There are $10$ vanishing principal $4$-minors, each contributing Euler characteristic one, and there are five vanishing principal $5$-minors, $l_2$ of them correspond to quadrics of rank $3$ in $\CC\PP^2$, each contributing Euler characteristic two; the remaining $(5-l_2)$ are quadrics of rank one, also contributing Euler characteristic two. Hence, $$\mldeg(V_A^c) = -[(-1)^5-(6+k_1)-3] = 10+k_1.$$ Since in this case $k=k_1$ and $r=5-l_2=5-(l-l_1) = 5-(l-1)=6-l$, it follows that $$\mldeg(V_A^c) = 10+k = 10+k+r-(6-l) = 4+k+l+r.$$ 

For the first case in Table~\ref{table:ML-degrees-order-6-rank-5}, where $k=0$, we note that there are $r=6-l$ principal $5$-submatrices of rank three. Hence, $$ 10 = 4+l+(6-l) = 4+l+r = 4+k+l+r.$$

It remains to prove the lower bound. By Table~\ref{table:number-of-delta-matroids-size-six-without-ground-set}, there exist no delta-matroids for $l\in[3]$. Since $\rank(C)=5$, we therefore have $l\ge 4$.

Now consider, for example, Case 3.6. We have $k\ge1$, since $k_1>0$. Moreover, if $l_2=4$, then there is a principal $5$-submatrix of rank three, since $k_2=0$, which implies $r\ge1$. If $l_2=5$, then necessarily $l\ge5$. In either situation, we obtain $4+k+l+r\ge10$. 

The remaining cases can be treated analogously. 
\end{proof}

\begin{table}[tbh] \centering 
\begin{tabular}{|| c || C{1cm} C{1cm} C{1cm} C{1cm} || c ||} 
 \hline
 Case & $k_1$ & $k_2$ & $l_1$ & $l_2$ & ML degree \\[0.5ex] 
 \hline\hline
 \multicolumn{6}{|c|}{all principal $5$-minors vanish} \\
 \hline\hline
 1 & $=5$ & $=10$ & $=0$ & $=0$ & $20$ \\ 
 \hline\hline
 \multicolumn{6}{|c|}{all principal $5$-minors do not vanish} \\
 \hline\hline
 2.1 & $=0$ & $=0$ & $=1$ & $=5$ & $11$ \\
 \hline
 2.2 & $=5$ & $=10$ & $=1$ & $=5$ & $26$ \\
 \hline
 2.3 & $\ge0$ & $\ge0$ & $=1$ & $=5$ & $11+k$ \\
 \hline\hline
 \multicolumn{6}{|c|}{some principal $5$-minors vanish} \\
 \hline\hline
 \multicolumn{6}{|c|}{all principal $4$-minors do not vanish} \\
 \hline\hline
 3.1.1 & $=5$ & $=10$ & $=0$ & $>0$ & $20+l$ \\
 \hline
 3.1.2 & $=5$ & $=10$ & $=1$ & $=0$ & $20+l$ \\
 \hline
 3.1.3 & $=5$ & $=10$ & $=1$ & $>0$ & $20+l$ \\
 \hline\hline
 \multicolumn{6}{|c|}{some principal $4$-minors do not vanish} \\
 \hline\hline
 3.2.1 & $=0$ & $>0$ & $=0$ & $>0$ & $5+k+l+r$ \\
 \hline 
 3.2.2 & $>0$ & $>0$ & $=0$ & $>0$ & $5+k+l+r$ \\
 \hline 
 3.2.3 & $=0$ & $>0$ & $=1$ & $>0$ & $5+k+l+r$ \\
 \hline 
 3.2.4 & $>0$ & $>0$ & $=1$ & $>0$ & $5+k+l+r$ \\
 \hline 
 3.2.5 & $>0$ & $=0$ & $=0$ & $>0$ & $5+k+l+r$ \\
 \hline 
 3.2.6 & $>0$ & $=0$ & $=1$ & $>0$ & $5+k+l+r$ \\
 \hline 
 3.2.7 & $=0$ & $>0$ & $=1$ & $=0$ & $5+k+l+r$ \\
 \hline 
 3.2.8 & $>0$ & $>0$ & $=1$ & $=0$ & $5+k+l+r$ \\
 \hline 
\end{tabular}

\caption{ML degrees of $V_6^c$ for scalings $c\in(\CC^*)^{15}$ with $\rank(C)=6$.} \label{table:ML-degrees-order-6-rank-6}
\end{table}

It remains to study scalings $c\in(\CC^*)^{15}$ whose associated scaling matrix has rank six.

\begin{proposition}
If $\rank(C)=6$, then $\mldeg(V_6^c)=5+k+l+r$. In particular, the ML degree is at least eleven.
\end{proposition}

\begin{proof}
Let $C\in\CC^{6\times6}$ be a scaling matrix of rank six, that is, $\det(C)\neq 0$. We consider the integers $k_i$, $l_i$, and $r_i$ for $i\in[2]$ as introduced in the proof of Proposition \ref{prop:ML-degree-order-6-rank-5}. 

Again, we analyze all possible patterns of vanishing principal $4$- and $5$-minors of $C$. The corresponding ML degrees are shown in Table~\ref{table:ML-degrees-order-6-rank-6}. Observe that the case where all principal $5$-minors vanish can only occur if all principal $4$-minors do not vanish, as follows from Table~\ref{table:number-of-delta-matroids-size-six-with-ground-set}.

The computation of the ML degrees, as well as the proof of the lower bound, were carried out analogously to the procedure in the proof of Proposition~\ref{prop:ML-degree-order-6-rank-5}. A case-by-case study shows that each ML degree listed in Table~\ref{table:ML-degrees-order-6-rank-6} can be written in the desired form $5+k+l+r$.
\end{proof}

This completes the proof of Theorem~\ref{thm:main-minimum}. We now deduce the delta-matroid invariance, thereby proving Theorem~\ref{thm:main-delta-matroid-invariance}. Recall that two delta-matroids are isomorphic if there exists a bijection between their ground sets that induces a bijection between their feasible sets. 

\begin{theorem} \label{thm:delta-matroid-invariance}
Let $2\le d\le6$. For $i\in[2]$, let $C_i\in\CC^{5\times5}$ be scaling matrices with associated scalings $c_i$. If the delta-matroids arising from $C_1$ and $C_2$ are isomorphic, then $$\mldeg(V_d^{c_1}) = \mldeg(V_d^{c_2}).$$
\end{theorem}

\begin{proof}
For $d=2$ and $d=3$, the statement is immediate. The case $d=4$ follows directly from Example~\ref{example:SecondHypersimplexOfOrderFour}. For $d=5$ and $d=6$, the preceding propositions show that the ML degree depends only on the vanishing pattern of the principal minors of the associated scaling matrix. Since isomorphic delta-matroids correspond to the same pattern of vanishing principal minors up to permutation of the ground set, the ML degrees coincide.
\end{proof}

\section{Discussion}
\label{sec:Discussion}

In this article, we proved \cite[Conjecture~3.9]{MatroidStratificationOfMLDegreesOfIndependenceModels} for $d\le6$. Our reasoning relies on explicit computations of the Euler characteristic, which in turn depend on the rank of the associated quadric. For second hypersimplices of higher order, however, determining explicit formulas for the ML degree becomes increasingly challenging using our current approach. 

The main problem remains to determine the minimum ML degree of $V_d^c$ for $d\ge 7$. Our results suggest that this problem is controlled by the underlying delta-matroids. However, it is poorly understood which subsets of $2^{[d]}$ actually form a delta-matroid. In this article, we relied on direct computations to determine the relevant delta-matroids. Unfortunately, the computational complexity grows rapidly with the dimension. Therefore, using our current brute-force implementation in \texttt{Macaulay2}, this enumeration seems infeasible for $d\ge7$.

Nevertheless, for rank-four cases for $d=7$ we excluded the existence of delta-matroids with few 4-subsets by computing orbit representatives, up to eight included $4$-subsets. Likewise, for ground sets $[4]$, $[5]$, and $[6]$, our computations indicate that at least $1$, $4$, and $10$ subsets of size four, respectively, must occur. We therefore state the following conjecture.

\begin{conjecture}
Let $C\in\CC^{d\times d}$ be a scaling matrix of rank four. The minimum number of non-vanishing principal $4$-minors is $$\binom{d-1}{3}.$$ In particular, the corresponding delta-matroid lies in the orbit that consists of families $$\cF_\alpha \coloneqq \{\emptyset\} \cup \binom{[d]}{2} \cup \binom{[d]}{3} \cup \cK_\alpha, \quad \alpha\in[d], \quad \textup{where } \; \cK_\alpha \coloneqq \left\{ K\in\binom{[d]}{4} \mid \alpha\in K \right\}.$$
\end{conjecture}

If the ML degree were known to be a delta-matroid invariant in general, the computational analysis could be reduced to a classification of delta-matroid orbits under the natural $S_d$-action. This perspective further suggests two additional directions for future research.

\bigskip

\noindent\textbf{Representability of Delta-matroids.} Little is known about the representability of delta-matroids by symmetric matrices via their non-vanishing principal minors. While all delta-matroids appearing in Table~\ref{table:number-of-delta-matroids-size-five} admit such a realization, it remains open whether every delta-matroid on ground set $[6]$ is representable. 
Consequently, it is unknown whether every ML degree arising from these vanishing patterns is actually realizable.

\bigskip

\noindent\textbf{ML Degree Monotonicity.} Recent work deals with the monotonicity of ML degrees across various settings. For instance, in Gaussian graphical models, the ML degree behaves monotonically on induced subgraphs \cite[Corollary 2.2]{OnTheMaximumLikelihoodDegreeForGaussianGraphicalModels}. A similar principle holds for scaled toric varieties where the ML degree is monotonic on the face poset \cite[Theorem~6]{TheMaximumLikelihoodDegreeOfToricModelsIsMonotonic}. Furthermore, in the context of two-way independence model, \cite[Conjecture~2.9]{MatroidStratificationOfMLDegreesOfIndependenceModels} suggests that the ML degrees respect the weak order on certain matroids. 
We observe an analogous behavior for delta-matroids arising from second hypersimplices.

\begin{example}
According to Table~\ref{table:number-of-delta-matroids-size-five}, there are 39 delta-matroids arising from scaling matrices of size five. Our computational analysis shows that these are grouped into nine orbits. We partially order these orbits by inclusion. The resulting Hasse diagram is shown in Figure~\ref{figure:orbits-poset-size-five}.

For ground set $[6]$, Table~\ref{table:number-of-delta-matroids-size-six-without-ground-set} and \ref{table:number-of-delta-matroids-size-six-with-ground-set} show that there are 155\,474 delta-matroids, grouped into 687 orbits. Table \ref{table:number-orbits-ML-degree-second-hypersimplex-order-six} shows the number of orbits corresponding to each possible ML degree of $V_6^c$. Using \texttt{Macaulay2}, we verified that the ML degree is monotonic on the associated~poset.

The computations of the orbits were carried out using the \texttt{Julia} package \texttt{GAP.jl}.

\begin{table}[tbh] \centering 
\begin{tabular}{|| c || C{0.9cm} C{0.9cm} C{0.9cm} C{0.9cm} C{0.9cm} C{0.9cm} C{0.9cm} C{0.9cm} C{0.9cm} ||} 
\hline
ML degree & 10 & 11 & 12 & 13 & 14 & 15 & 16 & 17 & 18 \\
\hline
Number of orbits & 4 & 4 & 7 & 15 & 25 & 37 & 59 & 82 & 109 \\[0.5ex] 
\hline \hline

ML degree & 19 & 20 & 21 & 22 & 23 & 24 & 25 & 26 & \\
\hline
Number of orbits & 116 & 102 & 66 & 35 & 15 & 7 & 3 & 1 & \\
\hline
\end{tabular}

\caption{Number of orbits for possible ML degrees of the second hypersimplex of order~six.} \label{table:number-orbits-ML-degree-second-hypersimplex-order-six}
\end{table}

\begin{figure}[tbh]
\centering
\includegraphics[width=0.3\textwidth]{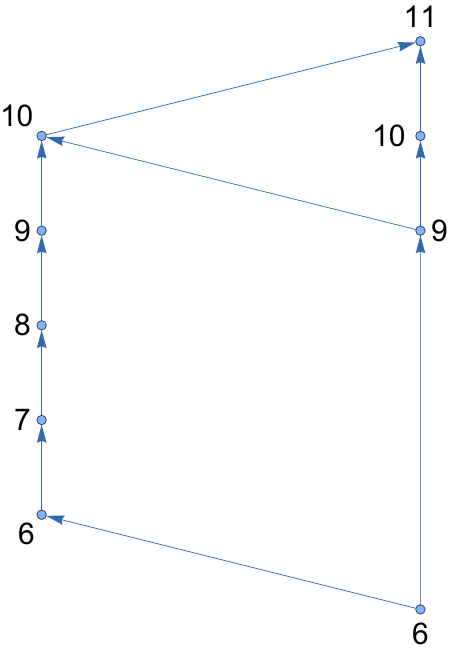}

\caption{Poset of delta-matroid orbits arising from scaling matrices of size five.} \label{figure:orbits-poset-size-five}
\end{figure}
\end{example}

Motivated by these computational results, we conjecture that the ML degree of second hypersimplices is monotonic on the poset of delta-matroid orbits ordered by inclusion.

\begin{acknowledgements*} \upshape
I thank Carlos Am\'endola and Serkan Ho\c{s}ten for many helpful conversations and valuable comments on earlier versions of the manuscript. I am especially grateful to Serkan for pointing out the connection to delta-matroids and \cite[Theorem~2.2]{EulerDiscriminantOfComplementsOfHyperplanes}.
\end{acknowledgements*}

\setlength{\bibsep}{0.1\baselineskip}
\bibliographystyle{alpha}
\bibliography{bib.bib}

\bigskip

\textsc{Technische Universität Berlin, Institute of Mathematics}

\textit{Email address:} {\tt oldekop@math.tu-berlin.de}

\end{document}